\documentclass{article}

\usepackage{amsmath}
\usepackage{amsthm}

\usepackage[pdftex]{hyperref}

\newtheorem{theorem}{Theorem}[section]

\numberwithin{equation}{section}

\begin{document}

\title{Some Combinatorial Identities some of which involving Harmonic Numbers}

\author{M.J. Kronenburg}
\date{}

\maketitle

\begin{abstract}
A product difference equation is proved and used for derivation
by elementary methods of four combinatorial identities,
eight combinatorial identities involving generalized harmonic numbers
and eight combinatorial identities involving classical harmonic numbers.
For the binomial coefficients the definition with gamma functions is used,
thus also allowing non-integer arguments in the identities.
The generalized harmonic numbers in this case are harmonic numbers with a complex offset,
where the classical harmonic numbers are a special case with offset zero.
\end{abstract}


\section{Introduction}

Binomial coefficients and their combination with harmonic numbers occur
frequently in applied mathematics \cite{GKP94,K97,L07,S90}.
Many ways to find and prove these identities can be found in literature
\cite{CY08,M02,PS03,PWZ96,Sip,WTS00}.
In this paper a product difference equation is used for derivation of
a number of such, often new, identities.

\section{A Product Difference Equation}

\begin{theorem}
Let $n$ be a nonnegative integer, and let $x$ and $y$ be complex variables
and let the set $\{z_k\}$ be $n$ complex variables
and the set $\{w_k\}$ be n complex variables.
Then there is the following product difference equation:
\begin{equation}\label{proddif}
\begin{split}
    & \prod_{k=1}^{n}(x-z_k)^{w_k}-\prod_{k=1}^{n}(y-z_k)^{w_k} \\
={} & \sum_{k=1}^{n}[(x-z_k)^{w_k}-(y-z_k)^{w_k}]
      \prod_{l=1}^{k-1}(x-z_l)^{w_l} \cdot \prod_{l=k+1}^{n}(y-z_l)^{w_l}
\end{split}
\end{equation}
where a sum of zero terms is zero and a product of zero terms is one.
\end{theorem}
\begin{proof}
Let $\alpha_k=(x-z_k)^{w_k}$ and $\beta_k=(y-z_k)^{w_k}$.
Then (\ref{proddif}) becomes:
\begin{equation}\label{proddifproof}
\begin{split}
 \prod_{k=1}^{n}\alpha_k - \prod_{k=1}^{n}\beta_k 
& = \sum_{k=1}^{n}(\alpha_k - \beta_k)
        \prod_{l=1}^{k-1}\alpha_l \cdot \prod_{l=k+1}^{n}\beta_l \\
& = \sum_{k=1}^{n}( \prod_{l=1}^{k}\alpha_l \cdot \prod_{l=k+1}^{n}\beta_l -
                  \prod_{l=1}^{k-1}\alpha_l \cdot \prod_{l=k}^{n}\beta_l )
\end{split}
\end{equation}
On the right side the first product for each $k$ cancels the second product
for each $k+1$, only leaving the first product for $k=n$ and the second product for $k=1$.
These two remaining products are the two products on the left side.
\end{proof}

In the case that all $z_k=0$ and all $w_k=1$, (\ref{proddif}) reduces to
the following power difference equation \cite{RW04}:
\begin{equation}\label{powdif}
  x^n - y^n = (x-y) \sum_{k=1}^{n}x^{k-1}y^{n-k}
\end{equation}

\section{Combinatorial Identities}

The definition of the binomial coefficient in terms of gamma functions
for complex $x$, $y$ is \cite{AS70,GS97,GKP94}:
\begin{equation}\label{binomdef}
 \binom{x}{y} = \frac{\Gamma(x+1)}{\Gamma(y+1)\Gamma(x-y+1)}
\end{equation}
For nonnegative integer $n$ and integer $k$ this reduces to \cite{AS70,G72,GKP94,K97}:
\begin{equation}\label{binomidef}
 \binom{n}{k} =
 \begin{cases}
   \dfrac{n!}{k!(n-k)!} & \text{if $0\leq k\leq n$} \\
   0 & \text{otherwise} \\
 \end{cases}
\end{equation}
From the recurrence formula for the gamma function for complex $s$ \cite{AS70,AAR99,GR07,RW04,WW06}:
\begin{equation}\label{gammadef}
 \Gamma(s+1) = s\,\Gamma(s)
\end{equation}
follows for integer $a\leq b$ and complex $s$:
\begin{equation}\label{combiprod}
 \prod_{k=a}^{b-1}(s-k) = \dfrac{\Gamma(s-a+1)}{\Gamma(s-b+1)}
    = (-1)^{b-a}\dfrac{\Gamma(b-s)}{\Gamma(a-s)}
\end{equation}
Combination of (\ref{proddif}) with the $\{z_k\}$ consecutive integers and all $w_k=w$
and (\ref{binomdef}) and one of the right side expressions of (\ref{combiprod}) for $x$
and for $y$ in (\ref{proddif}) yields the following four combinatorial identities.\\
For integer $a\leq b$ and complex $w$, $x$, $y$:
\begin{equation}\label{binomial1}
\begin{split}
 & \sum_{k=a}^{b-1}[(x+k+1)^w-(x-y+k)^w] \binom{x+k}{y}^w \\
={} & (y+1)^w \left[ \binom{x+b}{y+1}^w - \binom{x+a}{y+1}^w \right]
\end{split}
\end{equation}
\begin{equation}\label{binomial2}
\begin{split}
 & \sum_{k=a}^{b-1}[(y-x+k)^w-(y+k)^w] (-1)^{wk}\binom{x}{y+k}^w \\
={} & x^w \left[ (-1)^{wb}\binom{x-1}{y+b-1}^w - (-1)^{wa}\binom{x-1}{y+a-1}^w \right]
\end{split}
\end{equation}
\begin{equation}\label{binomial3}
\begin{split}
    & \sum_{k=a}^{b-1}[(-y-k-1)^w-(x-y-k+1)^w]
      (-1)^{wk}\binom{x}{y+k}^{-w} \\
={} & (x+1)^w
     \left[ (-1)^{wb}\binom{x+1}{y+b}^{-w} - (-1)^{wa}\binom{x+1}{y+a}^{-w} \right]
\end{split}
\end{equation}
\begin{equation}\label{binomial4}
\begin{split}
    & \sum_{k=a}^{b-1}[(x-k)^w-(y-k+1)^w]
      \binom{x}{k}^w \binom{y}{k}^{-w} \\
={} & (y+1)^w
     \left[ \binom{x}{b}^w \binom{y+1}{b}^{-w} -
            \binom{x}{a}^w \binom{y+1}{a}^{-w} \right]
\end{split}
\end{equation}
For $w=1$ these identities reduce to:
\begin{equation}\label{binomial1w1}
 \sum_{k=a}^{b-1} \binom{x+k}{y} = \binom{x+b}{y+1} - \binom{x+a}{y+1}
\end{equation}
\begin{equation}\label{binomial2w1}
 \sum_{k=a}^{b-1} (-1)^{k}\binom{x}{y+k} 
 = (-1)^{a}\binom{x-1}{y+a-1} - (-1)^{b}\binom{x-1}{y+b-1}
\end{equation}
\begin{equation}\label{binomial3w1}
 \sum_{k=a}^{b-1} (-1)^{k}\binom{x}{y+k}^{-1} 
 = \frac{x+1}{x+2} \left[ (-1)^{a}\binom{x+1}{y+a}^{-1} 
   - (-1)^{b}\binom{x+1}{y+b}^{-1} \right]
\end{equation}
\begin{equation}\label{binomial4w1}
 \sum_{k=a}^{b-1} \binom{x}{k}\binom{y}{k}^{-1} 
 = \frac{y+1}{x-y-1} \left[ \binom{x}{b} \binom{y+1}{b}^{-1} 
  - \binom{x}{a} \binom{y+1}{a}^{-1} \right]
\end{equation}
which are equivalent to known identities \cite{G72}.\\
For an example of $w=2$, let be given \cite{G72,K97}:
\begin{equation}\label{example1}
 \sum_{k=0}^{n} \binom{n}{k}^2 = \binom{2n}{n}
\end{equation}
Then (\ref{binomial2}) with $w=2$, $a=0$, $b=n+1$, $x=n$, $y=0$ and (\ref{example1}) yields:
\begin{equation}
 \sum_{k=0}^{n} k \binom{n}{k}^2 = \frac{n}{2} \binom{2n}{n}
\end{equation}
For an example of $w=3$, let a formula from A.C. Dixon be given \cite{G72,G04,G08,PWZ96}:
\begin{equation}\label{example2}
 \sum_{k=0}^{2n} (-1)^k \binom{2n}{k}^3 = (-1)^n \binom{2n}{n}\binom{3n}{n}
\end{equation}
Then (\ref{binomial2}) with $w=3$, $a=0$, $b=2n+1$, $x=2n$, $y=0$ and (\ref{example2}) yields:
\begin{equation}
 \sum_{k=0}^{2n} (-1)^k k(2n-k)\binom{2n}{k}^3 = (-1)^n \frac{4}{3}n^2\binom{2n}{n}\binom{3n}{n}
\end{equation}

\section{Combinatorial Identities with Harmonic Numbers}

The definition of the harmonic numbers for nonnegative integer $n$ is \cite{K97,L07}:
\begin{equation}\label{harmdef}
 H_n = \sum_{k=1}^{n} \frac{1}{k}
\end{equation}
from which follows that $H_0=0$.
The definition of the generalized harmonic numbers
for nonnegative integer $n$, complex order $m$ and
complex offset $c$ is \cite{LLF05,L07}:
\begin{equation}\label{genharmdef}
 H_{c,n}^{(m)} = \sum_{k=1}^{n}\frac{1}{(c+k)^m}
\end{equation}
from which follows that $H_{c,0}^{(m)}=0$.
There is the following symmetry formula:
\begin{equation}\label{genharmsym}
 H_{c,n}^{(m)} = (-1)^m H_{-(c+n+1),n}^{(m)}
\end{equation}
For order $m=1$ and nonnegative integer offset $c=p$ these numbers can be expressed
in the classical harmonic numbers of (\ref{harmdef}):
\begin{equation}\label{genharmidef}
 H_{p,n}^{(1)} = H_{p+n} - H_p
\end{equation}
from which follows:
\begin{equation}\label{genharmzero}
 H_{0,n}^{(1)} = H_n
\end{equation}
Identities involving these numbers can be obtained
by differentiation of the product (\ref{combiprod})
and using (\ref{genharmsym}):
\begin{equation}\label{difpow}
 \frac{d}{dx} \prod_{k=a}^{b-1}(x-k)^w
 = w H_{x-b,b-a}^{(1)} \prod_{k=a}^{b-1}(x-k)^w 
 = -w H_{a-x-1,b-a}^{(1)} \prod_{k=a}^{b-1}(x-k)^w
\end{equation}
and similarly for $y$.
These identities therefore are in pairs:
one from differentiation to $x$ and one from differentiation to $y$,
in each case using one of the right side expressions in (\ref{difpow}).
This results in the following eight combinatorial identities involving
generalized harmonic numbers.\\
For integer $a\leq b $ and complex $w$, $x$, $y$:
\begin{equation}\label{harmonic1a}
\begin{split}
    & \sum_{k=a}^{b-1}[(x+k+1)^w H_{x+a,k-a+1}^{(1)} - (x-y+k)^w H_{x+a,k-a}^{(1)}]
      \binom{x+k}{y}^w \\
={} & (y+1)^w \binom{x+b}{y+1}^w H_{x+a,b-a}^{(1)}
\end{split}
\end{equation}
\begin{equation}\label{harmonic1b}
\begin{split}
    & \sum_{k=a}^{b-1}[(x+k+1)^w H_{y-x-b,b-k-1}^{(1)} - (x-y+k)^w H_{y-x-b,b-k}^{(1)}]
      \binom{x+k}{y}^w \\
={} & -(y+1)^w \binom{x+a}{y+1}^w H_{y-x-b,b-a}^{(1)}
\end{split}
\end{equation}
\begin{equation}\label{harmonic2a}
\begin{split}
    & \sum_{k=a}^{b-1}[(y-x+k)^w H_{y-x+a-1,k-a+1}^{(1)} - (y+k)^w H_{y-x+a-1,k-a}^{(1)}] \\
    & \cdot  (-1)^{wk}\binom{x}{y+k}^w 
    = x^w (-1)^{wb}\binom{x-1}{y+b-1}^w H_{y-x+a-1,b-a}^{(1)}
\end{split}
\end{equation}
\begin{equation}\label{harmonic2b}
\begin{split}
    & \sum_{k=a}^{b-1}[(y-x+k)^w H_{-y-b,b-k-1}^{(1)} - (y+k)^w H_{-y-b,b-k}^{(1)}] \\
    & \cdot  (-1)^{wk}\binom{x}{y+k}^w 
    = -x^w (-1)^{wa}\binom{x-1}{y+a-1}^w H_{-y-b,b-a}^{(1)}
\end{split}
\end{equation}
\begin{equation}\label{harmonic3a}
\begin{split}
    & \sum_{k=a}^{b-1}[(-y-k-1)^w H_{y+a,k-a+1}^{(1)} - (x-y-k+1)^w H_{y+a,k-a}^{(1)}] \\
    & \cdot  (-1)^{wk}\binom{x}{y+k}^{-w} 
    = (x+1)^w (-1)^{wb}\binom{x+1}{y+b}^{-w} H_{y+a,b-a}^{(1)}
\end{split}
\end{equation}
\begin{equation}\label{harmonic3b}
\begin{split}
    & \sum_{k=a}^{b-1}[(-y-k-1)^w H_{x-y-b+1,b-k-1}^{(1)} - (x-y-k+1)^w H_{x-y-b+1,b-k}^{(1)}] \\
    & \cdot  (-1)^{wk}\binom{x}{y+k}^{-w}
    = -(x+1)^w (-1)^{wa}\binom{x+1}{y+a}^{-w} H_{x-y-b+1,b-a}^{(1)}
\end{split}
\end{equation}
\begin{equation}\label{harmonic4a}
\begin{split}
    & \sum_{k=a}^{b-1}[(x-k)^w H_{a-x-1,k-a+1}^{(1)} - (y-k+1)^w H_{a-x-1,k-a}^{(1)}]
      \binom{x}{k}^{w} \binom{y}{k}^{-w}\\
={} & (y+1)^w \binom{x}{b}^{w} \binom{y+1}{b}^{-w} H_{a-x-1,b-a}^{(1)}
\end{split}
\end{equation}
\begin{equation}\label{harmonic4b}
\begin{split}
    & \sum_{k=a}^{b-1}[(x-k)^w H_{y-b+1,b-k-1}^{(1)} - (y-k+1)^w H_{y-b+1,b-k}^{(1)}]
      \binom{x}{k}^{w} \binom{y}{k}^{-w}\\
={} & -(y+1)^w \binom{x}{a}^{w} \binom{y+1}{a}^{-w} H_{y-b+1,b-a}^{(1)}
\end{split}
\end{equation}
For $w=1$ these identities reduce to:
\begin{equation}\label{harmonic1aw1}
\begin{split}
    & \sum_{k=a}^{b-1}\binom{x+k}{y} H_{x+a,k-a}^{(1)} \\
={} & \binom{x+b}{y+1} ( H_{x+a,b-a}^{(1)} - \frac{1}{y+1} )
    + \frac{1}{y+1} \binom{x+a}{y+1}
\end{split}
\end{equation}
\begin{equation}\label{harmonic1bw1}
\begin{split}
    & \sum_{k=a}^{b-1}\binom{x+k}{y} H_{y-x-b,b-k-1}^{(1)} \\
={} & -\binom{x+a}{y+1} ( H_{y-x-b,b-a}^{(1)} - \frac{1}{y+1} ) 
    - \frac{1}{y+1} \binom{x+b}{y+1}
\end{split}
\end{equation}
\begin{equation}\label{harmonic2aw1}
\begin{split}
    & \sum_{k=a}^{b-1}(-1)^k\binom{x}{y+k} H_{y-x+a-1,k-a}^{(1)} \\
={} & (-1)^{b+1}\binom{x-1}{y+b-1} ( H_{y-x+a-1,b-a}^{(1)} + \frac{1}{x} )
    + \frac{(-1)^a}{x} \binom{x-1}{y+a-1}
\end{split}
\end{equation}
\begin{equation}\label{harmonic2bw1}
\begin{split}
    & \sum_{k=a}^{b-1}(-1)^k\binom{x}{y+k} H_{-y-b,b-k-1}^{(1)} \\
={} & (-1)^a\binom{x-1}{y+a-1} ( H_{-y-b,b-a}^{(1)} + \frac{1}{x} )
    - \frac{(-1)^b}{x} \binom{x-1}{y+b-1}
\end{split}
\end{equation}
\begin{equation}\label{harmonic3aw1}
\begin{split}
   & \sum_{k=a}^{b-1}(-1)^k\binom{x}{y+k}^{-1} H_{y+a,k-a}^{(1)} \\
={} & \frac{x+1}{x+2} \left[ (-1)^{b+1}\binom{x+1}{y+b}^{-1} ( H_{y+a,b-a}^{(1)} - \frac{1}{x+2} )
   - \frac{(-1)^a}{x+2} \binom{x+1}{y+a}^{-1} \right]
\end{split}
\end{equation}
\begin{equation}\label{harmonic3bw1}
\begin{split}
   & \sum_{k=a}^{b-1}(-1)^k\binom{x}{y+k}^{-1} H_{x-y-b+1,b-k-1}^{(1)} \\
={} & \frac{x+1}{x+2} \left[ (-1)^{a}\binom{x+1}{y+a}^{-1} ( H_{x-y-b+1,b-a}^{(1)} - \frac{1}{x+2} )
   + \frac{(-1)^b}{x+2} \binom{x+1}{y+b}^{-1} \right]
\end{split}
\end{equation}
\begin{equation}\label{harmonic4aw1}
\begin{split}
   & \sum_{k=a}^{b-1}\binom{x}{k}\binom{y}{k}^{-1} H_{a-x-1,k-a}^{(1)} = \frac{y+1}{x-y-1} \\
   & \cdot \left[ \binom{x}{b}\binom{y+1}{b}^{-1} ( H_{a-x-1,b-a}^{(1)} + \frac{1}{x-y-1} )
   - \frac{1}{x-y-1} \binom{x}{a}\binom{y+1}{a}^{-1} \right]
\end{split}
\end{equation}
\begin{equation}\label{harmonic4bw1}
\begin{split}
   & \sum_{k=a}^{b-1}\binom{x}{k}\binom{y}{k}^{-1} H_{y-b+1,b-k-1}^{(1)} = \frac{y+1}{y-x+1} \\
   & \cdot \left[ \binom{x}{a}\binom{y+1}{a}^{-1} ( H_{y-b+1,b-a}^{(1)} - \frac{1}{y-x+1} )
   + \frac{1}{y-x+1} \binom{x}{b}\binom{y+1}{b}^{-1} \right]
\end{split}
\end{equation}
Changing all $k$ into $a+b-k-1$ in the summation term reverses the direction of summation
and yields valid identities. When the offset of the generalized harmonic numbers in these
identities is zero, (\ref{harmonic1a}), (\ref{harmonic1aw1}), (\ref{harmonic3aw1}) and (\ref{harmonic4bw1})
yield one classical harmonic number identity
and three combinatorial identities involving classical harmonic numbers.\\
For nonnegative integer $n$ and complex $m$, $w$:
\begin{equation}\label{harmonictrad1}
 \sum_{k=0}^{n} \left[(k+1)^w-k^w\right] H_k = (n+1)^w H_{n+1} - H_{0,n+1}^{(1-w)}
\end{equation}
\begin{equation}\label{harmonictrad2}
 \sum_{k=0}^{n}\binom{k}{m}H_k = \binom{n+1}{m+1}(H_{n+1} - \frac{1}{m+1})
 + \frac{1}{m+1}\binom{0}{m+1}
\end{equation}
\begin{equation}\label{harmonictrad3}
 \sum_{k=0}^{n}\binom{m}{n-k}\binom{n}{k}^{-1}H_k 
 = \frac{n+1}{n-m+1} \left[ H_{n+1} + \frac{1}{n-m+1}(\binom{m}{n+1}-1) \right]
\end{equation}
\begin{equation}\label{harmonictrad4}
  \sum_{k=0}^{n}(-1)^k\binom{m}{k}^{-1}H_k 
 = \frac{m+1}{m+2} \left[ (-1)^n\binom{m+1}{n+1}^{-1}(H_{n+1} - \frac{1}{m+2})
   -\frac{1}{m+2} \right]
\end{equation}
The identity (\ref{harmonictrad1}) is an example of harmonic number
identities with polynomials or rational functions in $k$ \cite{SLL89,S90}.
The identity (\ref{harmonictrad2}) for nonnegative integer $m$
was already listed and proved in literature \cite{GKP94,K97}.
Using the same identities but now using (\ref{genharmidef}) yields the following
more general identity.\\
For nonnegative integer $n$, $p$ and complex $m$:
\begin{equation}\label{harmonicext2}
\begin{split}
   & \sum_{k=0}^n \binom{m}{n-k}\binom{n+p}{p+k}^{-1}H_{p+k} = \frac{n+p+1}{n+p-m+1} \\
   & \cdot \left[ H_{n+p+1}-\frac{1}{n+p-m+1}
  - \binom{m}{n+1}\binom{n+p+1}{p}^{-1} ( H_p - \frac{1}{n+p-m+1} ) \right]
\end{split}
\end{equation}
When $p=0$ this identity reduces to (\ref{harmonictrad3}).

\section{More Combinatorial Identities with Harmonic\\ Numbers}

More combinatorial identities involving harmonic numbers
may be derived from (\ref{binomial1w1}) to (\ref{binomial4w1})
using $d/dx\Gamma(x)=\Gamma(x)\psi(x)$ where $\psi(x)$ is the digamma function,
and using $\psi(x+n+1)-\psi(x+1)=H^{(1)}_{x,n}$ \cite{AAR99}.
Then the harmonic numbers are linked to binomial coefficients in the following way.\\
\begin{equation}
 \frac{d}{dx} H_{x+y,n}^{(m)} = -m H_{x+y,n}^{(m+1)}
\end{equation}
\begin{equation}
 \frac{d}{dx} \binom{x+y}{n}^w = w \binom{x+y}{n}^w H_{x+y-n,n}^{(1)} 
 = w \binom{x+y}{n}^w ( H_{x+y}-H_{x+y-n} )
\end{equation}
\begin{equation}
 \frac{d}{dx} \binom{n}{x+y}^w = w \binom{n}{x+y}^w H_{x+y,n-2(x+y)}^{(1)} 
 = w \binom{n}{x+y}^w ( H_{n-(x+y)}-H_{x+y} )
\end{equation}
Differentiating (\ref{binomial1w1}) to $y$ yields:
\begin{equation}
 \sum_{k=0}^n \binom{m+k}{m}H_k = \binom{n+m+1}{n}(H_n-\frac{1}{m+1})+\frac{1}{m+1}
\end{equation}
Differentiating (\ref{binomial2w1}) to $x$ yields:
\begin{equation}\label{more2}
 \sum_{k=0}^n (-1)^k\binom{n+m}{m+k}H_{m+k} = \binom{n+m}{n}(\frac{m}{n+m}H_m - \frac{n}{(n+m)^2})
\end{equation}
Differentiating (\ref{binomial2w1}) to $y$ and using (\ref{more2}) yields:
\begin{equation}\label{more3}
 \sum_{k=0}^n (-1)^k\binom{m}{k}H_k = (-1)^n \binom{m-1}{n}(H_n+\frac{1}{m})-\frac{1}{m}
\end{equation}
Differentiating (\ref{binomial3w1}) to $y$ and using (\ref{harmonictrad4}) yields:
\begin{equation}
\begin{split}
 & \sum_{k=0}^n (-1)^k\binom{n+m}{m+k}^{-1}H_{m+k} \\
 ={} &  \frac{n+m+1}{n+m+2} \left[ \binom{n+m+1}{m}^{-1} ( H_m-\frac{1}{n+m+2} )
  + (-1)^n ( H_{n+m+1}-\frac{1}{n+m+2} ) \right]
\end{split}
\end{equation}
Differentiating (\ref{binomial4w1}) to $x$ yields:
\begin{equation}
 \sum_{k=0}^n \binom{n}{k}\binom{n+m}{m+k}^{-1} H_k 
 = \frac{n+m+1}{m+1}\left[ H_n+\frac{1}{m+1}(\binom{n+m+1}{n}^{-1}-1) \right]
\end{equation}
The special cases $m=0$ of (\ref{more2}) and $m=n$ of (\ref{more3}) give the
well-known result for $n>0$ \cite{G61}:
\begin{equation}
 \sum_{k=0}^n (-1)^k \binom{n}{k} H_k = -\frac{1}{n}
\end{equation}

\pdfbookmark[0]{References}{}

\end{document}